\documentclass[a4paper,12pt]{article}
\usepackage{amssymb,amsthm,amsmath,amsfonts}
\usepackage{graphicx}
\usepackage[left=1.5cm,right=1.5cm,top=1.5cm,bottom=1.5cm]{geometry}
\usepackage{cite}
\usepackage{color}
 \newtheorem{thm}{Theorem}[section]
 \newtheorem{cor}[thm]{Corollary}
 \newtheorem{lem}[thm]{Lemma}
 \newtheorem{prop}[thm]{Proposition}
 \newtheorem{defn}[thm]{Definition}
 \newtheorem{rem}[thm]{Remark}
 \numberwithin{equation}{section}

\begin{document}
\title{On slice regular fractal-fractional Dirichlet type spaces}

\small{
\author
{Jos\'e Oscar Gonz\'alez-Cervantes$^{(1)}$, Carlos Alejandro Moreno-Mu\~noz$^{(1)}$\\and\\ Juan Bory-Reyes$^{(2)\footnote{corresponding author}}$}
\vskip 1truecm
\date{\small $^{(1)}$ Departamento de Matem\'aticas, ESFM-Instituto Polit\'ecnico Nacional. 07338, Ciudad M\'exico, M\'exico\\ Email: jogc200678@gmail.com; camoreno@ipn.mx\\ $^{(2)}$ {SEPI-ESIME-Zacatenco-Instituto Polit\'ecnico Nacional. 07338, Ciudad M\'exico, M\'exico}\\Email: juanboryreyes@yahoo.com}

\maketitle
\begin{abstract}
In this paper, we study some families of right modules of quaternionic slice regular functions induced by a generalized fractal-fractional derivative  with respect to a truncated quaternionic exponential function on slices. Important Banach spaces of slice regular functions, namely the Bergman and the Dirichlet modules, are two important elements of one of our families.
\end{abstract}


\section{Introduction} 
The aim of this paper is to combine the fractal-fractional calculus and the systematic study of Dirichlet type spaces in the theory of quaternionic slice regular functions of quaternionic variable. Thus, we describe properties of some fractal-fractional families of quaternionic right module of slice regular functions. 

Fractional calculus is a theory allowing integrals and derivatives of arbitrary real or complex order. The interest in the subject has been growing continuously during the last few decades because of numerous applications in diverse fields of science and engineering. For classical treatment of the theory we refer the reader to \cite{DMP, H, JAR, Je, KH, KST, MKM, SKM, Ro, OS, P}.

The fractal-fractional derivative (a new class of fractional derivative, which has many applications in real world problems) is a mathematical concept that combines two different ideas: fractals and fractional derivatives. Fractals are complex geometric patterns that repeat at different scales, while fractional derivatives are a generalization of ordinary derivatives that allow for non-integer orders. The combination of fractal theory and fractional calculus gave rise to new concepts of differentiation and integration. 

Functions slice regular were introduced in 2006 by Gentili and Struppa \cite{GeSt}, inspired by a paper by Cullen \cite{Cu}, in the case of quaternionic valued functions defined on the unit ball of the quaternions, identified with $\mathbb R^4$.

Let $\Omega$ be a suitable open subset of the space of quaternions $\mathbb H$ that intersects the real line and let $\mathbb S^2$ be the unit sphere of purely imaginary quaternions. Such functions are also called slice regular in the case they are defined on quaternions and are quaternionic-valued. Slice regular functions are those functions $f : \Omega \rightarrow \mathbb H$, whose restriction to the complex planes $\mathbb C({\bf i})$, for every $\bf i \in \mathbb S^2$, are holomorphic
maps. One of their crucial properties is that from the knowledge of the values of $f$ on $\Omega \cap \mathbb C({\bf i})$ for some $\bf i \in \mathbb S^2$, one can reconstruct $f$ on the whole $\Omega$ by the so called Representation Formula.

The literature on slice-regularity is wide, and we refer the reader for the generalities on this function theory to \cite{CCGG, GGJ, cssbook, CGS3, CSS, newadvances, CGSBERgman, CGLSS, csTrends, CSS2, GS} and the references therein.

Recently (see \cite{GS, GB1, GB2, GBS, GBNS, GMB}) the development of an extension of quaternionic analysis (in particular slice regular functions theory) to the fractional case started.

To the best of our knowledge, there is no definition made in the literature on the complex generalized fractal-fractional derivative. Hence, this paper will introduce and study some properties of a family of complex slice regular fractal-fractional Dirichlet-type spaces associated to a generalized fractal-fractional derivative with respect to a complex truncated exponential function.

The plan of the paper is as follows. Section \ref{Sec2} presents some preliminaries on the quaternionic slice regular function theory. Section \ref{Sec3} contains a brief summary of basic notions of real generalized fractal-fractional derivative. In Section \ref{Sec4} we introduce the notion of a complex generalized fractal-fractional derivative. In the final Section \ref{Sec5} our main results are stated and proved. Indeed, properties of a family of complex slice regular fractal-fractional Dirichlet-type spaces associated to a generalized fractal-fractional derivative with respect to a complex truncated exponential function are discussed. 

\section{Preliminaries on quaterionic slice regular functions}\label{Sec2}
The skew field of quaternions $\mathbb H$ consists of $q=q_0+q_1e_1+q_2e_2+q_3e_3$, where 
 $ q_n\in \mathbb R$ for all $n$   and   $e_1$, $e_2$ and  $e_3$ satisfy:
 $e_1^2=e_2^2=e_3^2=-1$,  $e_1e_2=e_3$, $e_2e_3=e_1$, $e_3e_1= e_2$. 
The  real part of the quaternion $q$ is $Re(q)=q_0$ and its vector part   is  $\mathbf{q}=q_1e_1+q_2e_2+q_3e_3\mapsto  (q_1, q_2, q_3) \in \mathbb R^3$. In addition, the  conjugate quaternion of $q$ is $\bar q = q_0-{\bf q}$ and its quaternionic  module is   
$|q|=\sqrt{q_0^2+q_1^2+q_2^2+q_3^2}$, see \cite{KS}.  

The unit sphere in $\mathbb R^3$  is given by $\mathbb S^2:=\{ {\bf q} \in\mathbb R^3 \ \mid  \  | {\bf q} |=1 \}$  and  the unit open ball in $\mathbb R^4$ is $\mathbb B:=\{ {q} \in \mathbb H \ \mid   \  |  q | <1 \}$. 
 
The quaternionic algebra allows to see that   ${\bf{i}}^2=-1$  for all  ${\bf{i} }\in \mathbb S^2  $. Therefore,  
 $\mathbb C({\bf i}):=\{ x+{\bf i} y \mid \ x,y\in\mathbb R\} \cong \mathbb C$ as fields. 

\begin{defn} 
A domain  $\Omega \subseteq \mathbb{H}$  is called   axially symmetric s-domain  if    $\Omega \cap \mathbb R\neq \emptyset$,      $x+{\bf i}y \in \Omega$, where $x,y\in \mathbb R$ implies that   
$x+{\bf j}y\in\Omega$ for all ${\bf{j}}\in \mathbb{S}^2$  and   
 $ \Omega_{\bf i}:=\Omega\cap\mathbb C({\bf i})$ is a domain in $\mathbb C({\bf i})$ for all ${\bf i} \in \mathbb{S}^2$. 
See  \cite{csTrends, cssbook, gssbook}. In addition,  a real differentiable function 
$f:\Omega\to\mathbb{H}$    is called quaternionic   slice regular function, or slice regular function, if 
\begin{align*}
 \left( \frac{\partial}{\partial x} + {\bf i}  \frac{\partial}{\partial y}\right) f_{\mid_{\Omega_{{\bf i}}}}=0 , \  \textrm{  on }  \  \Omega_{{\bf i}} ,\end{align*} 
 for any ${\bf i}\in \mathbb S^2$ and  its derivative, or Cullen derivative  is $\displaystyle f' =\frac{\partial f}{\partial x}$. By  
 $\mathcal{SR}(\Omega)$ we  denote  the quaternionic right-module  of slice regular functions on $\Omega$. 
 
 Given  $f\in \mathcal{SR}(\Omega)$ then $f$ is called intrinsic slice regular function,  if $f(q)=\overline{f(\bar q)}$ for all $q\in\Omega$,  see  \cite{GS}. By    $\mathcal N(\Omega)$ we mean the real-linear space formed by all  intrinsic slice regular functions on $\Omega$. In particular,  the quaternionic exponential funcion 
$$e^q:=\sum_{k=0}^\infty \frac{q^k}{k!}, \quad \forall  q\in \mathbb H.$$ 
Even more,  if $f\in \mathcal N(\Omega)$ then $\textrm{exp}[f] = e^f $ it is well-defined.
\end{defn}

\begin{rem}
 Given $\beta\in (0,1]$  the mapping  
$$q\mapsto q^{\beta}=  e^{\beta (\ln (q))} = e^{\beta (\ln |q|+ I_q \textrm{Arg}(q))} = \sum_{n=0}^{\infty }   \frac{1}{n!}(\ln |q|+ I_q \textrm{Arg}(q))^n \alpha^n $$ is defining for all $q\in \mathbb H\setminus \{0\}$, where the identities $q=a+ I_q b =|q| e^{ I_q \textrm{Arg}(q) }$ are used  in the complex plane $\mathbb C(I_q)$.
  
In addition, the   quaternionic exponential truncated with order $k \in \mathbb N\cup\{0\} $  is given by the sum 
$e_k(q) :=\displaystyle \sum_{n=0}^k \frac{q^n}{n!}$   for  all  $q\in \mathbb  H$ and for $k=\infty$ we will denote the quaternionic exponential function  $e_{\infty}(q) =:e^q$ for all $ q\in \mathbb H$. Note that  $e_k\in\mathcal {SR}(\mathbb H) $ for all $k\in\mathbb N\cup\{0\}\cup \{\infty \}$ and  we have that $e_0' (q) =0$,  $e_{k}'(q)= e_{k-1}(q)$,   $e_{\infty}'(q) = e_{\infty}(q)$ for all $q\in\mathbb H $ and $k\in \mathbb N$.  
\end{rem}

In \cite{csTrends, cssbook, gssbook} we can see Splitting Lemma and  Representation Formula Theorem. 
\begin{lem}\label{LemS}
(Splitting Lemma) If $f\in \mathcal {SR}(\Omega)$ and ${\bf i}\in \mathbb S^2$  then   $f_{\mid\Omega_{\bf i}} =f_1 +f_2  {\bf j}$ 
 where $f_1,f_2\in Hol( \Omega_{\bf i})$ and ${\bf j} \in \mathbb S^2$  is orthogonal to  ${\bf i}$.
\end{lem}

\begin{thm}\label{TeoR}
(Representation Formula Theorem) Given $f\in \mathcal {SR}(\Omega)$ and  ${\bf{i}}, {\bf j} \in \mathbb S^2$ one has that 
$ f(x+y{\bf{i}})=\frac{1}{2}(1-{\bf i} {\bf j}) f(x+y{\bf j})+ \frac{1}{2}(1+{\bf i} {\bf j}) f(x-y{\bf j})$,
 for all $x+y{\bf{i}}\in \Omega$, where $x,y\in \mathbb R$.
\end{thm}

\begin{defn}  Let $\Omega\subset \mathbb H$ be an axially symmetric s-domain. 
 Define	$ P_{\bf i} :   Hol(\Omega_{\bf i})+ Hol(\Omega_{\bf i}){\bf j} \longrightarrow  \mathcal{SR}(\Omega)$  defined as 	$$  P_{\bf i}[f](x+yI_q)=\frac{1}{2}\left[(1+ I_q{\bf i})f(x-y{\bf i}) + (1- I_q {\bf i}) f(x+y {\bf i})\right],$$ 		  where $x,y\in \mathbb R$ and $I_q\in \mathbb S^2$. 
\end{defn}
					
\begin{rem}	If $\Omega=\mathbb B$ let us recall that given  
 $f,g\in \mathcal {SR}(\mathbb B)$ there exist   two sequences of quaternions  $(a_n)$ and $(b_n)$ 
such that
$f(q) = \sum_{n=0}^{\infty} q^n a_n$,  $   g(q) = \sum_{n=0}^{\infty} q^n b_n$. Then  $f*g$ is defined as   
$f*g(q) := \sum_{n=0}^{\infty} q^n \sum_{k=0}^n  a_k b_{n-k}$   for all $q\in \mathbb B$. Denote   
$f^{*n}= f*\cdots * f $, ($n$-times) for 	any $n\in \mathbb N $. 	 Moreover,  if $f(q)\neq 0 $ then 
$f * g(q)= f(q)g(f(q)^{-1} qf(q))$, 
see \cite{cssbook, gssbook}.  In addition, denote 
$f^c (q)=  \sum_{n=0}^{\infty} q^n \overline{a_n} $, for all $q\in \mathbb B$,      
$f^s =  f* f^c =  f^c * f $ and if the zero set  of $f^s$ is the empty set then the  $*$-inverse of $f$
is given by $f^{-*} = \frac{1}{f^s} * f^c  $ and   $(f^{-*})'= - f^{-*} * f'* f^{-*} $, see \cite{cssbook, csTrends}.
\end{rem}

\section{On a generalized fractal derivative}\label{Sec3}
Paper \cite{KH}  study  the following new class of generalized differential operators used to describe the behaviour  of many phenomena in terms  of  complex dynamic systems.  
\begin{defn}\label{fractal} 
Let $I\subset \mathbb R$ be an interval. The fractal derivative of  $f\in C(I,\mathbb R)$ with respect to a fractal measure $\nu:(0,\infty ) \times I \to \mathbb R$ is given  by
\begin{align*} \frac{d_\nu f(t)}{
dt^\eta} := \lim_{\tau \to t} 
\frac{
f(t) - f(\tau)}
{\nu(\eta, t) - \nu (\eta, \tau)},  
\end{align*}
if it exists for all $ t \in   I$, $f$ will be called  real fractal differentiable on $I$ with order $\eta$. Sometimes, $\frac{d_\nu f(t)}{dt^\eta}$ is called the Stieltjes derivative, see \cite{A}.
\end{defn}

\begin{rem} \label{rem0}
According to the previous definition we have the following cases:
\begin{enumerate}
\item If   $\nu(\eta, t) = t$ for all $t\in I$ then $\dfrac{d_\nu  f}{
dt^\eta} = \dfrac{d f}{
dt } $ is the usual derivative of $f$.  
\item If  $\nu (\eta, t) = t^\eta$ for all $t\in I$ then   $\dfrac{d_\nu f(t)}{
dt^\eta}$  is the well-known Hausdorff derivative.  
 \item Let   $\nu(h, t) = h(t)$ for all $t\in I$ where $h'(t) > 0$ for all $t\in I$  then 
$ \dfrac{d_\nu f(t)}{
dt^\eta}  =  \dfrac{f'(t)
}{h'(t)} $ for all $f\in C^1(I)$. For example, case  $\nu (\alpha, t)= e^{t^\alpha}$ for all $t\in I$ and $\alpha \in (0,1]$
is studied in \cite{AS}.  
\end{enumerate}
\end{rem}
 
The following differential operator is  another extension of the  previous fractal derivative.
\begin{defn}\label{defbfractal} Given $\beta\in [0,1]$ we present  the  $\beta$-fractal derivative of  $f\in C(I,\mathbb R)$  with respect   $\nu $ as follows;   
\begin{align*} \frac{d^{\beta}_\nu f(t)}{
dt^\eta} := \lim_{\tau \to t} 
\frac{
(f(t))^{\beta} - (f(\tau))^{\beta}}
{\nu(\eta, t) - \nu (\eta, \tau)}, \quad  \eta > 0. 
\end{align*}
\end{defn}

In \cite{JAR} is presented the proportional derivative which is deeply related with the tempered derivative.
\begin{defn}\label{defchi}
Given   $ \chi_0 ,  \chi_1\in C( [0, 1]\times I ,\mathbb R) $   such that  
$$\lim_{\sigma\to 0^+} \chi_1(\sigma,t) = 1,  \lim_{\sigma\to 0^+} \chi_0(\sigma,t) = 0, \lim_{\sigma\to 1^-} \chi_1(\sigma,t) = 0, \lim_{\sigma\to 1^-} \chi_0(\sigma,t) = 1.$$
The proportional derivative of $f\in C^1(I)$ of order $\sigma\in [0,1]$ is given by
\begin{align*} 
D^{\sigma}f (t) =  \chi_1(\sigma,t)f (t) +  \chi_0(\sigma,t)\frac{d f} {dt} (t) ,\quad \forall  t\in I.  
\end{align*}
In particular,   for  
$   \chi_1(\sigma,t) = 1-\sigma $ and $\chi_0(\sigma,t) = \sigma$ we see that  $
D^{\sigma}f $  is the convex combination of $f$ and $\dfrac{d f} {dt} $.
\end{defn}
  
Combining  Definitions \ref{defbfractal} and \ref{defchi} we propose the following operator:
\begin{defn}\label{defpropbfractal} 
Let $I\subset \mathbb R$ be an interval and   $\beta\in [0,1]$. Consider the fractal measure $\nu:(0,\infty ) \times I \to \mathbb R  $  and given   $ \chi_0 ,  \chi_1\in C( [0, 1]\times I,\mathbb R) $ be as Definition \ref{defchi}. The $\beta$-fractal fractional derivative of   $f\in C(I,\mathbb R)$   with respect to $\nu(\eta, t)$ and $\sigma$ is given by
\begin{align*}
\frac{d^{\sigma,\beta}_\nu f(t)}{
dt^\eta}  (t) =  \chi_1(\sigma,t)f (t) +  \chi_0(\sigma,t)
\frac{d^{\beta}_\nu f(t)}{
dt^\eta}, 
\end{align*}
if it exists for all $t\in I$.
\end{defn}

\begin{rem}\label{rem1}     
Given $\alpha \in (0,1] $ and  $k\in \mathbb N\cup\{0\}$. The $k$-truncated exponential function is given by 
$\displaystyle e_k(t^{\alpha}) := \sum_{n=0}^k \frac{(t^\alpha)^n}{n!}$ for all $t\in \mathbb R$.   
Note that $e_0(t^{\alpha})= 1 $, $e_1(t^{\alpha})= 1+ t^{\alpha}$,  and  $e_\infty(t^{\alpha}) = e^{t^{\alpha}}  $,  for all  $t\in I$. 
We can see that $\displaystyle e_k(t)' =  e_{k-1}(t)$ for all $t\in \mathbb R$ and   $k\in \mathbb N$.  

Moreover, in paper \cite{AS} the truncated exponential function is used as a  fractal measure. In particular,  the fractal derivative  with respect to the fractal measure $\nu(1, t) = e_1(t^{\alpha}) $ 
coincides with the Hausdorff derivative given by $\nu (\alpha , t) = t^\alpha$ for all $t\in I$ in second fact of  Remark \ref{rem0}. 
\end{rem}

\begin{defn}\label{defpropbfractalexp} 
Let $I\subset \mathbb R$ be an interval. Given $\alpha \in (0,1]$, $k\in \mathbb N\cup\{0\}$  and $\beta\in [0,1]$. Consider the     fractal measure $\nu(k, t) = e_k(t^{\alpha})$ and let $\chi_0, \chi_1\in C([0, 1]\times I,\mathbb R) $ be as Definition \ref{defchi}.

The $\beta$-fractal fractional derivative of $f\in C(I,\mathbb R)$ with respect to $e_k(t^{\alpha})$ is given by
\begin{align*}
\frac{ {d}^{\sigma,\beta}}{ d t _{\alpha, k} }f (t) : = \frac{d^{\sigma,\beta}_\nu f(t)}{
dt^k}  (t) =  \chi_1(\sigma,t)f (t) +  \chi_0(\sigma,t)
\frac{d^{\beta}_\nu f(t)}{dt^k}, 
\end{align*}
if it exists for all $t\in I$.  In addition,  if $f\in C^1(I,\mathbb R)$ and its zero set is the empty set  then 
\begin{align*} 
\frac{ {d}^{\sigma,\beta}}{dt_{\alpha, k}}f(t) = &  \chi_1(\sigma,t)f(t) + \chi_0(\sigma,t) \frac{(f^{\beta})'(t)}{e _k( t^{\alpha})'   }  \\ 
= &  \chi_1(\sigma,t) f (t) + \chi_0(\sigma,t) \frac{ \beta (f (t))^{\beta-1 }f'(t)}{\alpha t^{\alpha-1}  e_{k-1} ( t^{\alpha})    } ,\end{align*}
for all $t\in I$.
\end{defn}

\begin{rem}\label{rem2}
If $\chi_1(\sigma,t) = 1-\sigma $ and  $\chi_0(\sigma,t)=\sigma$  for all $\sigma\in [0,1]$ and  $t\in I$ in the previous derivative,  then
\begin{align*}  
\frac{ {d}^{\sigma,\beta}}{ d t _{\alpha, k} }f (t): = &  (1-\sigma)f (t) + \sigma \frac{(f ^{\beta})'(t)}{e_k ( t^{\alpha})'}  \\ 
: = &  (1-\sigma)f (t) +   \sigma \frac{ \beta (f (t))^{\beta-1 }f'(t)}
{ \alpha t^{\alpha-1}  e_{k-1 } ( t^{\alpha})    } , \quad \forall \ t \in  I, 
\end{align*}
for all  $f\in C^1(I)$. 

Moreover, cases $k=1,\infty $ give us that 
\begin{align*} \frac{ {d}^{\sigma,\beta}}{dt _{\alpha, 1}}f(t) = (1-\sigma)f (t) + \sigma \frac{\beta (f (t))^{\beta-1 }f'(t)}{\alpha t^{\alpha-1}},\\
\displaystyle \frac{ {d}^{\sigma,\beta}}{ d t _{\alpha, \infty} }f (t)  = (1-\sigma)f (t) +   \sigma \frac{  \beta (f (t))^{\beta-1 }    f'(t)}{\alpha t^{\alpha-1} e^{t^{\alpha}}}, \quad f\in C^1(I),
\end{align*}
where $\alpha\neq 0$ is a  necessary condition. But, the parameter $\sigma$ can play an interesting roll since choosing $\sigma=\alpha$ we obtain that  
\begin{align*}  
\frac{ {d}^{\alpha,\beta}}{dt _{\alpha, 1} }f (t) = & (1-\alpha)f (t) + \frac{ (f ^{\beta})'(t)}{t^{\alpha-1}},\\
\displaystyle \frac{ {d}^{\alpha,\beta}}{dt _{\alpha, \infty}}f (t) = & (1-\alpha)f (t) + \frac{(f^{\beta})'(t)}{t^{\alpha-1} e^{t^{\alpha}}   } , \quad  \forall f\in C^1(I). 
\end{align*}
So condition $\alpha=0$ can be ommited in this differential operators. Even more, if $\sigma=\alpha^2 $ we have that     
\begin{align*}  \frac{ {d}^{\alpha^2,\beta}}{ d t _{\alpha, 1} }f (t)  = & (1-\alpha^2)f (t) +   \alpha t^{1-\alpha} (f ^{\beta}) '(t)  ,\\
 \displaystyle \frac{ {d}^{\alpha^2,\beta}}{ d t _{\alpha, \infty} }f (t)  = & (1-\alpha^2)f (t) +    \alpha   t^{1-\alpha} e^{- t^{\alpha}}    (f ^{\beta}) '(t)   
    \end{align*}
and for  $\alpha=0,1$ and $k=1,\infty $ we have  
\begin{align*} 
\dfrac{ {d}^{0,\beta}}{ d t _{0, 1} }f   =& f  , \quad 
 \dfrac{ {d}^{1,\beta}}{ d t _{1, 1} }f   =   (f^\beta) ' , \\
  \dfrac{ {d}^{0,\beta}}{ d t _{0, \infty} }f   = & f, \quad 
  \frac{ {d}^{1,\beta}}{ d t _{1, \infty} }f  (t) =   e^{- t }    (f ^{\beta}(t)) ',   \quad  \forall t\in I ,
  \end{align*}
   for all $f\in C^1(I)$.
\end{rem}

\section{On a complex generalized fractal-fractional derivative}\label{Sec4}
Recall that given a domain $ G \subset \mathbb  C$ and a locally bounded family of holomorphic functions  $(f_n)_{n\in \mathbb N}$ defined  on $G$. If there exists a  holomorphic function  $f$ on $G$ such that  the set  
$\left\{z \in  G \ \mid \  {\displaystyle \lim_{n\rightarrow \infty } f_n(z) }= f(z)\right\}$ has a limit point in $G$ then   
 Vitali's theorem shows that   $(f_n)$ converges to $f$ uniformally on compact subsets of $G$, see \cite{Con}.

By $\mathbb D$ we denote the unit disk in the complex plane. 
Given $\alpha\in (0,1]$ recall  that  
$z^{\alpha}= e^{\alpha (\ln |z|+ i \textrm{Arg}(z))}$ for all $z\in \mathbb C\setminus \{0\} $  where  $-\pi <\textrm{Arg}(z)\leq \pi$ and   $\mathbb D\setminus (-1,0]$ is a simply connected domain in wich the function $z^{\alpha}$ is well-defined 
 
In the following definition the  mapping  $z\mapsto e_k(z^{\alpha}) $ will be  the  complex  fractal   function  associated to  $\mathbb D\setminus (-1,0]$.
   
\begin{defn} Given $\alpha \in (0,1] $, $\beta\in [0,1]$.   Consider $k\in \mathbb N\cup\{0\}$  such that mapping $e_{k-1}(z^{\alpha}) \neq 0$ for all $z\in \mathbb D\setminus (-1,0]$. Given   $ \chi_0 ,  \chi_1\in C( [0, 1]\times (\mathbb D\setminus (-1,0]),\mathbb C) $ such that 
$$\lim_{\sigma\to 0^+} \chi_1(\sigma,z) = 1,  \lim_{\sigma\to 0^+} \chi_0(\sigma,z) = 0, \lim_{\sigma\to 1^-} \chi_1(\sigma,z) = 0, \lim_{\sigma\to 1^-} \chi_0(\sigma,z) = 1,$$
for all $z\in \mathbb \mathbb D\setminus (-1,0]$. 
   
The generalized fractal-fractional derivative $\dfrac{ {d}^{\sigma,\beta}}{ d z _{\alpha, k} }$ acts on $f\in \textrm{Hol}
(\mathbb D\setminus (-1,0])$ such that $Z_f \neq \emptyset $, the zeros set   of $f$ is the empty set, as follows:    
\begin{align*}  
\frac{ {d}^{\sigma,\beta}}{dz_{\alpha, k} }f (z) := & \chi_1(\sigma,z) f(z) + \chi_0(\sigma,z) \lim_{\zeta \to z} \frac{ f ^{\beta}(z) - f^{\beta} (\zeta)}{e _k( z^{\alpha})- e _k( \zeta^{\alpha})},
 \\
= &  \chi_1(\sigma,z) f (z) +  \chi_0(\sigma,z) \frac{ (f ^{\beta})'(z)}{e _k( z^{\alpha})'} \\
=&  \chi_1(\sigma,z)f (z) +   \chi_0(\sigma,z) \frac{\beta (f(z))^{\beta-1} f'(z) }{\alpha  z^{ \alpha-1} e_{k-1} (z^{\alpha})},
\quad \forall z\in\mathbb D\setminus (-1,0].
\end{align*}
\end{defn}

\begin{defn} 
Given $\alpha \in (0,1] $, $\beta\in [0,1]$, $k\in \mathbb N\cup\{0\}$ and $\chi_0, \chi_1\in C( [0, 1]\times (\mathbb D\setminus (-1,0]),\mathbb C) $. The  generalized fractal-fractional Dirichlet type space of holomorphic functions induced by 
$\dfrac{{d}^{\sigma,\beta}}{dz_{\alpha, k}}$, denoted by $\mathcal D^{\sigma,\beta}_{\alpha, k}$, consists of $f\in \textrm{Hol}(\mathbb D\setminus (-1,0]  )$ such that  $Z_f \neq \emptyset $ and  
\begin{align*} 
\left( \|f\|_{\mathcal D^{\sigma,\beta}_{\alpha, k}}\right)^2 := & \alpha \|f(\frac{1}{2})\|^2 + \int_{\mathbb D\setminus (-1,0]}  \|\frac{ {d}^{\sigma,\beta}}{dz_{\alpha, k} }f (z)\|^2 d_z\mu <\infty.
\end{align*}  
\end{defn}

\begin{rem}
The complex exponential functions require the complex logarithm function. We will then consider our calculations on the main branch. Therefore $\frac{{d}^{\sigma,\beta}}{dz _{\alpha, k}} f \in \textrm{Hol}(\mathbb D\setminus (-1,0])$ for all $f\in \textrm{Hol}(\mathbb D\setminus  (-1,0] )$ such that $Z_f \neq \emptyset$.
\end{rem}

\begin{defn} 
In particular, when $\beta=1$ the differential operator $\frac{{d}^{\sigma, 1}}{dz_{\alpha, k}}$ is a complex linear operator and the condition $Z_f \neq \emptyset$ is not mandatory. So from now on, denote $\frac{ {d}^{\sigma}}{dz_{\alpha, k}}:= \frac{ {d}^{\sigma, 1}}{ dz_{\alpha, k} }$ and $\mathcal D^{\sigma }_{\alpha, k}:= \mathcal D^{\sigma,1}_{\alpha, k}$. 
Define  
\begin{align*}   
\langle f,g \rangle_{\mathcal D^{\sigma}_{\alpha, k}} :=  &  \alpha f( \frac{1}{2})\overline{g(  \frac{1}{2})} +  \int_{\mathbb D\setminus (-1,0]} \left(\frac{ {d}^{\sigma}}{dz_{\alpha, k}} f(z) \right) \overline{\left( \frac{ {d}^{\sigma}}{ d z _{\alpha, k} } g(z) \right)} d_z\mu    
,\end{align*}
and 
 \begin{align*}   
\|f\|_{\mathcal D^{\sigma }_{\alpha, k}}  :=   &  \sqrt{ \alpha \|f(\frac{1}{2})\|^2  +  \int_{\mathbb D\setminus  (-1,0] }  \| \frac{ {d}^{\sigma}}{ d z _{\alpha, k} } f(z) \|^2 d_z\mu }  
,\end{align*}
for all $f,g\in \mathcal D^{\sigma }_{\alpha, k}$. 
\end{defn}

From now on  to simplify the computations we assume$ \chi_0(\sigma,z) = \sigma$ and $\chi_1(\sigma,z) =1-\sigma$.
  
\begin{prop}\label{pro1} If $\sigma\in (0,1]$ and $f\in \text{Hol}(\mathbb D\setminus (-1,0])$, then 
\begin{align*}
\frac{d}{dz} \left(\textrm{exp}[\dfrac{\sigma-1}{\sigma}  e_k (z^{\alpha}) ] f(z) \right) 
= & \frac{1}{\sigma}  \textrm{exp}[\dfrac{\sigma-1}{\sigma}  e_k (z^{\alpha}) ]\left(\frac{d}{dz}    e_k (z^{\alpha}) \right)     \frac{ {d}^{\sigma}}{ d z _{\alpha, k} } f(z), 
\end{align*}
 for all $z\in \mathbb D\setminus (-1,0]$.
\end{prop}
\begin{proof}
Direct computations.
\end{proof}

Next computations use that $\alpha, \sigma\in (0,1]$ and  $k\in \mathbb N$ unless otherwise noted. 
 
\begin{prop}\label{pro2}
The function set $(\mathcal D^{\sigma }_{\alpha, k}, \|\cdot \|_{\mathcal D^{\sigma}_{\alpha, k}}  )$ is a  normed complex  linear space. 
 \end{prop}
\begin{proof} Given  $f\in \mathcal D^{\sigma }_{\alpha, k}$ such that 
  \begin{align*}  
0 = \|f\|_{\mathcal D^{\sigma }_{\alpha, k}} = \sqrt{\alpha \|f(\frac{1}{2})\|^2 + \int_{\mathbb D\setminus (-1,0]}  \| \frac{ {d}^{\sigma}}{dz_{\alpha, k}} f(z) \|^2 d_z\mu}.
 \end{align*}
Therefore, $f(\dfrac{1}{2} )=0 $ and  $\dfrac{ {d}^{\sigma }}{dz_{\alpha, k} } f(z) =0 $ for all $z\in \mathbb D\setminus (-1,0]$. From Proposition \ref{pro1} there exists a constant $c\in \mathbb C$ such that   
$$f(z) = c \, \textrm{exp}[ \dfrac{ 1- \sigma }{\sigma}  e_k (z^{\alpha}) ], \quad \forall z\in \mathbb D\setminus (-1,0].$$   Therefore, $c=0$ and  $f \equiv 0$ on $\mathbb D\setminus (-1,0]$. 
 
Finally, direct computations allow us to see that  
\begin{align*}
\|f \|_{\mathcal D^{\sigma }_{\alpha, k} }  \geq  & 0 , \\ 
\|a f  \|_{\mathcal D^{\sigma }_{\alpha, k} } = & |a| \|f \|_{\mathcal D^{\sigma }_{\alpha, k} } \\
\|f+g \|_{\mathcal D^{\sigma }_{\alpha, k} }  \leq & \|f \|_{\mathcal D^{\sigma }_{\alpha, k} }   + \|g \|_{\mathcal D^{\sigma }_{\alpha, k} }  ,  
\end{align*}
for all $f,g \in  \mathcal D^{\sigma }_{\alpha, k}$ and $a\in \mathbb C$.
\end{proof}

\begin{prop} \label{pro3}  
$(\mathcal D^{\sigma}_{\alpha, k}, \|\cdot \|_{\mathcal D^{\sigma}_{\alpha, k}})$ is a complex Banach space. 
 \end{prop}
\begin{proof}
Let $(f_n) $ be a Cauchy sequence of elements of $ \mathcal D^{\sigma }_{\alpha, k}$. For any $\epsilon > 0 $ there exists  $N\in \mathbb N$ such that 
 \begin{align*}       
\sqrt{\alpha \|f_n(\frac{1}{2})- f_m(\frac{1}{2})\|^2  +  \int_{\mathbb D\setminus (-1,0] }  \| \frac{ {d}^{\sigma}}{ d z _{\alpha, k} } f_n(z) - \frac{ {d}^{\sigma}}{ d z _{\alpha, k} } f_m(z) \|^2 d_z\mu } \ < & \ \epsilon,   
\end{align*}
 for all  $n,m> N$. Therefore, $(f_n(\frac{1}{2}))$ is a Cauchy sequence of complex numbers and  there exists $w\in \mathbb C$ such that  
 $\displaystyle \lim_{n \to \infty } f_n(\frac{1}{2})=w$. Also there exists $g\in L_2(\mathbb D\setminus (-1,0] ,\mathbb C)$ such that 
the sequence $(\dfrac{ {d}^{\sigma}}{ d z _{\alpha, k} } f_n )$ converges to $g$ in $L_2(\mathbb D\setminus (-1,0],\mathbb C)$.  

\noindent
In addition, $\dfrac{ {d}^{\sigma }}{ d z _{\alpha, k} } f _n \in \textrm{Hol}(\mathbb D\setminus (-1,0] )$ for all $n\in \mathbb N$ and for any compact subset $C \subset \mathbb D\setminus (-1,0]$  Cauchy formula allows us to obtain a  constant $\lambda _C>0$ sucht that 
$$\sup\{|\dfrac{ {d}^{\sigma }}{dz_{\alpha, k}} f_n (z)| \ \mid  \ z\in K \} \leq  \lambda_C  \sqrt{\int_{\mathbb D\setminus (-1,0]} | \frac{ {d}^{\sigma }}{ d z _{\alpha, k} } f_n(z) |^2 d_z\mu } \leq \lambda_C \|f \|^{\sigma }_{\alpha, k}.$$
Then  there exists $h\in \textrm{Hol}(\mathbb D\setminus (-1,0] ) $  such that 
$(\dfrac{ {d}^{\sigma }}{ d z _{\alpha, k} } f_n )$ converges to $h$ unifomally on compact subsets of $\mathbb D\setminus (-1,0]$ and using Proposition \ref{pro1} we see that the sequence     
$$\dfrac{d}{dz} \left(\textrm{exp}[\dfrac{\sigma-1}{\sigma} e_k (z^{\alpha})] f_n(z) \right), \quad z\in \mathbb D\setminus (-1,0]$$ 
converges to  
$$l(z) = \dfrac{1}{\sigma}  \textrm{exp}[ \dfrac{\sigma-1}{\sigma} e_k (z^{\alpha})] \left(\dfrac{d}{dz} e_k (z^{\alpha})\right) h(z), \quad z\in \mathbb D\setminus (-1,0]$$  
uniformly on compact subsets of $\mathbb D\setminus (-1,0]$. 

On the other hand, let $K \subset \mathbb D\setminus [0,1)$ be a compact set such that $\frac{1}{2}\in K $. There exist $M > 0$ and  
$N\in \mathbb N$ such that $|l(z)| \leq M$, 
\begin{align*}     
|\dfrac{d}{dz} \left(\textrm{exp}[ \dfrac{\sigma-1}{\sigma}  e_k (z^{\alpha}) ] f_n(z) \right) - l(z)| < & 1, \\
   |\textrm{exp}[ \dfrac{\sigma-1}{\sigma}  e_k ((\frac{1}{2})^{\alpha}) ] f_n(\frac{1}{2})| < & M,  
\end{align*}
for all $z\in K$ and  $n\geq N$. Then,   
\begin{align*}
|\dfrac{d}{dz} \left(\textrm{exp}[\dfrac{\sigma-1}{\sigma} e_k (z^{\alpha})] f_n(z) \right)| \leq & | \dfrac{d}{dz} \left(\textrm{exp}[ \dfrac{\sigma-1}{\sigma}  e_k (z^{\alpha}) ] f_n(z) \right) - l(z) |  + |l(z)| \\
 < &   1 + M,   
\end{align*} 
for all $z\in K$ and  $n\geq N$. Let $\gamma$ be a rectifiable path contained in $K$ and from $\dfrac{1}{2}$ to any $z\in K$ we see that   
\begin{align*} 
&  | \textrm{exp}[ \dfrac{\sigma-1}{\sigma}  e_k (z^{\alpha}) ] f_n(z) |
\\
\leq &  | \textrm{exp}[ \dfrac{\sigma-1}{\sigma}  e_k (z^{\alpha}) ] f_n(z) -   
\textrm{exp}[ \dfrac{\sigma-1}{\sigma}  e_k ((\frac{1}{2})^{\alpha}) ] f_n(\frac{1}{2})
|  +   
|\textrm{exp}[ \dfrac{\sigma-1}{\sigma}  e_k ((\frac{1}{2})^{\alpha}) ] f_n(\frac{1}{2})
|  \\ 
=  &  | \int_{\gamma}  \dfrac{d}{d\zeta } \left(\textrm{exp}[ \dfrac{\sigma-1}{\sigma}  e_k (\zeta  ^{\alpha}) ] f_n(\zeta  ) \right)
   d\zeta  |  +    
|\textrm{exp}[ \dfrac{\sigma-1}{\sigma}  e_k ((\frac{1}{2})^{\alpha}) ] f_n(\frac{1}{2})
|  \\
 \leq  & \int_{\gamma}   | \dfrac{d}{d\zeta} \left(\textrm{exp}[ \dfrac{\sigma-1}{\sigma}  e_k (\zeta ^{\alpha}) ] f_n(\zeta) \right)|| d\zeta|  + 
|\textrm{exp}[ \dfrac{\sigma-1}{\sigma}  e_k ((\frac{1}{2})^{\alpha}) ] f_n(\frac{1}{2})
| \\
    \leq &   ( 1 + M  ) V(\gamma) +    
M , \quad \forall n\geq N, 
   \end{align*} 
where $V(\gamma)$ is the length of $\gamma$. Therefore,   
\begin{align*} 
   | f_n(z) |
  \leq   | \textrm{exp}[ \dfrac{1-\sigma}{\sigma}  e_k (z^{\alpha}) ]  |  \left(  (1 + M  ) V(\gamma)  + M  \right) \leq A   \left(  (1 + M  ) V(\gamma)  + M  \right) , 
   \end{align*} 
for all $
z\in K$ and $ n\geq N$,    
where $ | \textrm{exp}[ \dfrac{1-\sigma}{\sigma}  e_k (z^{\alpha}) ]  |\leq A $ for all $
z\in K$, i.e.,  $(f_n)$ is a locally bounded family of holomorphic functions. In addition, 
 let $p\in \textrm{Hol}(\mathbb D\setminus (-1,0] ) $ be a primitive of $l$. 
 Then   for any rectifiable path $\gamma$ contained in $\mathbb D\setminus (-1,0] $ from $\dfrac{1}{2}$ to any $z\in\mathbb D\setminus (-1,0] $ we see that   
\begin{align*} & \lim_{n\to \infty }  \left(\textrm{exp}[ \dfrac{\sigma-1}{\sigma}  e_k (z^{\alpha}) ] f_n(z) - \textrm{exp}[  \dfrac{\sigma-1}{\sigma}  e_k ((\dfrac{1}{2})^{\alpha}) ] f_n(\dfrac{1}{2})  \right) \\
=  & 
\lim_{n\to \infty } \int_{\gamma} \dfrac{d}{d\zeta} \left(  \textrm{exp}[  \dfrac{\sigma-1}{\sigma}  e_k (\zeta^{\alpha}) ] f_n(\zeta)  \right)  d\zeta  =
 \int_{\gamma} p'(\zeta)d\zeta = p(z)- p(\dfrac{1}{2}) .  
\end{align*} 
From the previous puntual convergence we define: 
 \begin{align*}  & \lim_{n\to \infty }   f_n(z)   =  \textrm{exp}[ \dfrac{1-\sigma}{\sigma}  e_k (z^{\alpha}) ]  \left[   p(z)+  \lambda \right] =:f(z) , \quad \forall z\in  \mathbb D\setminus (-1,0] ,
\end{align*} 
where   
$$\lambda = - p(\dfrac{1}{2}) + \textrm{exp}[\dfrac{\sigma-1}{\sigma} e_k ((\dfrac{1}{2})^{\alpha})] w.$$ 
Vitali's theorem implies that $(f_n)$ converges to $f$ uniformly on compact subsets of $\mathbb D\setminus (-1,0]$. Then $( \frac{{d}^{\sigma }}{ d z _{\alpha, k} } f_n)$ converges to $\frac{ {d}^{\sigma}}{dz_{\alpha, k}} f$ uniformly on compact subsets of $\mathbb D\setminus (-1,0]$ and as a direct consequence we see that $\frac{ {d}^{\sigma }}{dz_{\alpha, k}} f = g$, a.e. on $\mathbb D\setminus (-1,0]$, i.e.,  $f\in \mathcal D^{\sigma }_{\alpha, k}$.
\end{proof}

\begin{rem}
From direct computations we see that $ \langle  \cdot ,\cdot  \rangle_{\mathcal D^{\sigma }_{\alpha, k}} $ is a complex inner product in  $  \mathcal D^{\sigma }_{\alpha, k}  $ and as a consequence     $\left(  \mathcal D^{\sigma }_{\alpha, k}  ,  \langle  \cdot ,\cdot  \rangle_{\mathcal D^{\sigma }_{\alpha, k}} \right) $ is a complex  Hilbert space.
\end{rem}
The following statement shows the behaviour of the expansion series of elements of $\textrm{Hol}(\mathbb D) \cap \mathcal D^{\sigma }_{\alpha, k}$.

\begin{prop}\label{pro6} 
 If     $f  \in \textrm{Hol}(\mathbb D) \cap \mathcal D^{\sigma }_{\alpha, k}$ such that 
 $$f(z)= \sum_{n=0}^\infty z^n a_n ,\quad \forall z\in \mathbb D,$$
 then  
 \begin{align*}  & \|  f\|^2_{\mathcal D^{\sigma }_{\alpha, k}}  =  \alpha   |f(\frac{1}{2})|^2 + (1-\sigma)^2 \pi \sum_{n=0}^\infty  \frac{1}{n+1} | a_n|^2     
  \\
  &  +  \frac{\sigma^2 }{\alpha^2 }  \sum_{n,m=0}^\infty  (n+1)(m+1) a_{n+1} \bar a_{m+1}             \lim_{
 {   \begin{array}{c} 
  \rho \to 1^-   \\ 
  \theta_1\to -\pi^+\\
  \theta_2 \to  \pi^- \end{array}
  } }    \int_{0}^\rho \int_{\theta_1}^{\theta_2}  
           \frac{  r ^{  n+m+3 -2\alpha } e^{i \theta (n-m)} } { | e_{k-1} ( r^{\alpha} e^{ i \alpha \theta}) |^2    }        dr d\theta 
    \\ 
&  +
 \frac{  (1-\sigma) \sigma }{\alpha}   \sum_{n,m=0}^\infty 2 (m+1) \Re  \left(   a_{m+1}   \bar  a_n          \lim_{
 {  \begin{array}{c} 
  \rho \to 1^-   \\ 
  \theta_1\to -\pi^+\\
  \theta_2 \to  \pi^- \end{array}
  } }    \int_{0}^\rho \int_{\theta_1}^{\theta_2}         \frac{ r^{n+m+ 2-\alpha }  e^{ i\theta ( m +1 -n - \alpha)} }{    e_{k-1} (  r^{\alpha} e^{ i \theta \alpha})    }     dr d\theta
 \right) .    \end{align*}
\end{prop}
\begin{proof}
From direct computations we see that  
\begin{align*}  & |(1-\sigma)  f (z)   +   \sigma \frac{f '(z) }{ \alpha z^{\alpha-1} e_{k-1}(z^\alpha)} |^2 \\
    =   &    (1-\sigma)^2  \sum_{n,m=0}^\infty    a_n \bar a_m      z^n   {(\bar z)}^m     
  \\
  &  +  \frac{\sigma^2 }{\alpha^2 }  \sum_{n,m=0}^\infty  (n+1)(m+1) a_{n+1} \bar a_{m+1}              \frac{ z^{n+1-\alpha }  {(\bar z)}^{m+1-\alpha } }{ | e_{k-1} ( z^{\alpha}) |^2    }      d_z\mu    \\ 
&  +  (1-\sigma) \sigma \frac{1}{\alpha}   \sum_{n,m=0}^\infty  (m+1) \left(\bar a_{m+1} a_n \frac{z^n {(\bar z)}^{m+ 1-\alpha }}{ e _{k-1}( {(\bar z)}^{\alpha})} +  a_{m+1} \bar a_n \frac{z^{m+ 1-\alpha } {(\bar z)}^n}{e_{k-1} ({z}^{\alpha}) } \right). 
\end{align*}
Consider $z=r e^{i\theta}$ where $0<r <\rho <1$ and $-\pi<\theta_1 <\theta <\theta_2 < \pi$ to compute  the integration.
\end{proof}

\begin{cor}\label{corK1Infty}
According to the previous proposition consider $f\in  \textrm{Hol}(\mathbb D)$ and 
$$f(z)= \sum_{n=0}^\infty z^n a_n ,\quad \forall z\in \mathbb D,$$
then we have the following identities:

\begin{enumerate}
\item 
If $f\in \mathcal D^{\sigma }_{\alpha, 1}$, then  
 \begin{align*}  & \|  f\|^2_{\mathcal D^{\sigma }_{\alpha, 1}} = \alpha |f(\frac{1}{2})|^2 + (1-\sigma)^2 \pi \sum_{n=0}^\infty  \frac{1}{n+1} | a_n|^2     
   +  \frac{\sigma^2 }{\alpha^2 }  \sum_{n =0}^\infty  \frac{(n+1)^2}{ 2n +4-2\alpha} |a_{n+1}| ^2  \\
   &  +
 \frac{  (1-\sigma) \sigma }{\alpha}   \sum_{n,m=0}^\infty 
 \frac{ 2 (m+1) }{ (n+m+3-\alpha) (m-n+1-\alpha)} \Im \left[ (e^{i2\pi (m-n+1 -\alpha)}-1) a_{m+1} \bar a_n \right].    
\end{align*}

\item  
If $f\in \mathcal D^{\sigma }_{\alpha, \infty}$,  then  
\begin{align*}  & \|  f\|^2_{\mathcal D^{\sigma }_{\alpha, \infty}}  =  \alpha   |f(\frac{1}{2})|^2 + (1-\sigma)^2 \pi \sum_{n=0}^\infty  \frac{1}{n+1} | a_n|^2     
  \\
  &  +  \frac{\sigma^2 }{\alpha^2 }  \sum_{n,m=0}^\infty  (n+1)(m+1) a_{n+1} \bar a_{m+1}          
  \int_{0}^1 \int_{-\pi}^{ \pi}\frac{r^{  n+m+3 -2\alpha } e^{i \theta (n-m)}} {\textrm{exp} (2r^{\alpha} \cos  (\alpha \theta))}        dr d\theta 
    \\ 
&  +
 \frac{  (1-\sigma) \sigma }{\alpha} \sum_{n,m=0}^\infty 2 (m+1) \Re  \left(   a_{m+1}   \bar  a_n         
   \int_{0}^1 \int_{-\pi}^{ \pi} \frac{ r^{n+m+ 2-\alpha} e^{ i\theta ( m +1 -n - \alpha)} }{\textrm{exp} (r^{\alpha} e^{ i \theta \alpha}) } dr d\theta \right).    
\end{align*}
\end{enumerate}
\end{cor}

\begin{proof}
Follows from direct computations and  note  that $k=1$ implies  $e_{k-1} (w) =1$, and  if $k=\infty$ then $e_{k-1} (w) =\textrm{exp}(w) =e^w$ for all $w\in \mathbb C$.
\end{proof}

\begin{prop}\label{CReprodProp} Reproducing property. 
Suppose  $\sigma\in (0,1)$ and $f\in \textrm{Hol}(\mathbb D)$ such that $\dfrac{ {d}^{\sigma}}{ d z _{\alpha, k} } f $ belongs to the complex  Bergman space associated to $\mathbb D$. Then
\begin{align*}
& f (z)   
 = -  \frac{ \sigma }{1-\sigma  }\frac{  f   '(z) }{ \alpha  z^{ \alpha-1}  e_{k-1} ( z^{\alpha})     } +  \frac{ 1 }{1-\sigma  }  \int_{\mathbb D} B(z,\zeta)  \frac{ {d}^{\sigma}}{ d z _{\alpha, k} } f(\zeta) d_{\zeta}\mu 
 \end{align*} 
  and 
  \begin{align*}
 &  f(z) =  \textrm{exp}[\dfrac{\sigma-1}{\sigma}  ( \  e_k ((\frac{1}{2})^{\alpha}) - e_k (z^{\alpha}) \  ) ] f(\frac{1}{2}) +  \int_{\mathbb D} \mathcal K_{\frac{1}{2}}(z,\zeta)   \frac{ {d}^{\sigma}}{ d z _{\alpha, k} } f(\zeta) d_{\zeta}\mu, 
\end{align*}
where
$$ \mathcal K_{\frac{1}{2}}(z,\zeta) =  \textrm{exp}[\dfrac{1-\sigma}{\sigma}  e_k (z^{\alpha}) ]  \int_{\gamma_{\frac{1}{2}, z}}  \frac{1}{\sigma}  \textrm{exp}[\dfrac{\sigma-1}{\sigma}  e_k (w^{\alpha}) ]\left(\frac{d}{dw}    e_k (w^{\alpha}) \right)     B(w,\zeta) dw , $$ 
 $B(\cdot, \cdot)$ is the Bergman kernel associated to $\mathbb D$ and $\gamma_{\frac{1}{2}, z}$ is a rectifiable path contained in $\mathbb D \setminus (-1,0]$ from 
  $\frac{1}{2}$ to  $z$.
\end{prop}

\begin{proof} The reproducing property of the Bergman kernel  gives us directly the  first  identity since 
\begin{align*}
& (1-\sigma )f (z) +    \sigma \frac{  f   '(z) }{ \alpha  z^{ \alpha-1}  e_{k-1} ( z^{\alpha})     }  
 =  \int_{\mathbb D} B(z,\zeta)  \frac{ {d}^{\sigma}}{ d z _{\alpha, k} } f(\zeta) d_{\zeta}\mu .
 \end{align*}
 On the other hand, using Proposition \ref{pro1} we see that  
 \begin{align*}
& \frac{d}{dz} \left(\textrm{exp}[\dfrac{\sigma-1}{\sigma}  e_k (z^{\alpha}) ] f(z) \right) \\
=&  \int_{\mathbb D} \frac{1}{\sigma}  \textrm{exp}[\dfrac{\sigma-1}{\sigma}  e_k (z^{\alpha}) ]\left(\frac{d}{dz}    e_k (z^{\alpha}) \right)     B(z,\zeta)  \frac{ {d}^{\sigma}}{ d z _{\alpha, k} } f(\zeta) d_{\zeta}\mu. 
\end{align*} 
Therefore, 
\begin{align*} 
& \textrm{exp}[\dfrac{\sigma-1}{\sigma}  e_k (z^{\alpha}) ] f(z) =  \textrm{exp}[\dfrac{\sigma-1}{\sigma}  e_k ((\frac{1}{2})^{\alpha}) ] f(\frac{1}{2}) \\ 
 &+  \int_{\mathbb D} \left(  \int_{\gamma_{\frac{1}{2}, z}}  \frac{1}{\sigma}  \textrm{exp}[\dfrac{\sigma-1}{\sigma}  e_k (w^{\alpha}) ]\left(\frac{d}{dw}    e_k (w^{\alpha}) \right)     B(w,\zeta) dw\right) \frac{ {d}^{\sigma}}{ d z _{\alpha, k} } f(\zeta) d_{\zeta}\mu. 
\end{align*}
\end{proof}
 
\begin{rem}\label{RemDirchlet-Bergman}
If we chose $\sigma=\alpha^2$ then we obtain the family of function spaces 
  $(\mathcal D^{\alpha^2}_{\alpha, k} )_{\alpha \in (0,1]}$ and its respective fractal fractional derivative is given by
  \begin{align*}  \frac{ {d}^{\alpha^2 }}{ d z _{\alpha, k} }f (z)  = (1-\alpha^2)f (z) +   \alpha z^{\alpha -1 }\frac{ f   '(z) }{   e ( z^{\alpha})_{k -1}    },
 \quad \forall   f\in C^1(\mathbb D ,\mathbb C ),  
\end{align*}
for all $z\in \mathbb D$.  Note the  fact  $\alpha =0$  is allowed  
\begin{align*}  \frac{ {d}^{0 }}{ d z _{0, k} }f (z)  =  f (z), \quad \forall  f\in C^1(\mathbb D ,\mathbb C )    \end{align*} 
for all $z\in \mathbb D$,  although we probably do not have many of the results demonstrated above. In addition for $\alpha =1$ we see that  
\begin{align*}  \frac{ {d}^{1 }}{ d z _{1, k} }f (z)  =   \frac{ f   '(z) }{   e ( z )_{k -1}    },
 \quad \forall   f\in C^1(\mathbb D ,\mathbb C ),  
\end{align*}
for all $z\in \mathbb D$. 

Therefore, in case  $k=1$  we see that  the family of function spaces    $(\mathcal D^{\alpha^2 }_{\alpha, 1} )_{\alpha \in [0,1]}$ contains the Bergman space, $\mathcal D^{0 }_{0, 1} $, and the Dirichlet space, $\mathcal D^{1}_{1, 1}$, both associated to $\mathbb D\setminus(-1,0]$.
\end{rem}

\section{Quaternionic slice regular fractal-fractional Dirichlet type spaces}\label{Sec5}
\begin{defn} 
Given $\alpha \in (0,1]$, a quaternionic fractal measures is  given in terms of the quaternionic exponential truncated as follows: 
$q\mapsto e_k(q^{\alpha})$ for $k\in \mathbb N\cup\{0\}$ such that mapping $e_{k-1}(q^{\alpha}) \neq 0$ for all $q\in  \mathbb E := \mathbb B\setminus (-1,0]$, where $(-1,0]$ consists of $q$ such that  ${\bf q}$ is the vector zero and $ -1 < q_0 \leq 0 $.  

Consider the mappings $\chi_0, \chi_1\in C((0, 1) \times \mathbb E, \mathbb R)$ such that 
$$\lim_{\sigma\to 0^+} \chi_1(\sigma,q) = 1, \lim_{\sigma\to 0^+} \chi_0(\sigma,q) = 0, \lim_{\sigma\to 1^-} \chi_1(\sigma,q) = 0, \lim_{\sigma\to 1^-} \chi_0(\sigma,q) = 1,$$
for all $q\in   \mathbb E$. 
Given $\beta \in (0,1] $ and  ${\bf i}\in \mathbb S^2$ then the  quaternionic  generalized fractal-fractional  ${\bf i}$-derivative of $f\in \mathcal{SR}(\mathbb E)$, the zero set of $f\mid_{\mathbb E_{\bf i}}$ is the empty set, is given   as follows: 
\begin{align*}
& \frac{ {d}^{\sigma,\beta}}{ d q _{\alpha, k  , {\bf i} } } f\mid_{\mathbb E_{\bf i}} (z)\\
: = & \chi_1(\sigma, z)  f\mid_{\mathbb E_{\bf i}} (z) + \chi_0(\sigma, z) \lim_{\zeta\to z} \left(e _k( z^{\alpha}) - e _k( \zeta^{\alpha}) \right)^{-1} \left(f\mid_{\mathbb E_{\bf i}} ^{\beta} (z) - f\mid_{\mathbb E_{\bf i}} ^{\beta} (\zeta) \right),   
\end{align*}
for all $z\in\mathbb E_{\bf i}$.   
\end{defn} 
    
\begin{prop}
The quaternionic generalized fractal-fractional ${\bf i}$-derivative of $f\in \mathcal{SR}(\mathbb E)$ such that $f ^{\beta}\in  \mathcal{SR}(\mathbb E)$ and the zero set of $f\mid_{\mathbb E_{\bf i}}$ is the empty set  is given by  
\begin{align*}
& \frac{ {d}^{\sigma,\beta}}{ d q _{\alpha, k , {\bf i} } } f\mid_{\mathbb E_{\bf i}} (z)\\
= & \chi_1(\sigma, z) f\mid_{\mathbb E_{\bf i}} (z) + \chi_0(\sigma,z)(\alpha z^{\alpha-1} \ast e _k( z^{\alpha})')^{-\ast} \ast (f\mid_{\mathbb E_{\bf i}} ^{\beta})' (z), \quad \forall z\in\mathbb E_{\bf i},
\end{align*}
where $(f ^{\beta})'$ and $e _k( q^{\alpha})'$ are Cullen derivatives of the slice regular functions $f^{\beta}$ and $e _k( q^{\alpha})$, respectively. 
\end{prop}

\begin{proof}
\begin{align*}
& \frac{{d}^{\sigma,\beta}}{dq_{\alpha, k}}_{\bf i}f\mid_{\mathbb E_{\bf i}} (z) = \chi_1(\sigma, z)  f\mid_{\mathbb E_{\bf i}} (z)  \\
& + \chi_0(\sigma, z)  \lim_{\zeta\to z} \left(e _k( z^{\alpha}) - e _k( \zeta^{\alpha}) \right)^{-1} (z-\zeta)^{-1}(z-\zeta) \left(f\mid_{\mathbb E_{\bf i}} ^{\beta}  (z) -  f\mid_{\mathbb E_{\bf i}} ^{\beta}  (\zeta) \right)  \\
= & \chi_1(\sigma, z)   f\mid_{\mathbb E_{\bf i}} (z)   +   \chi_0(\sigma,z)   ( e _k( z^{\alpha})' ) ^{-\ast}   \ast  (f\mid_{\mathbb E_{\bf i}} ^{\beta})' (z), \quad \forall z\in\mathbb E_{\bf i} \\
= & \chi_1(\sigma, z)   f\mid_{\mathbb E_{\bf i}} (z) + \chi_0(\sigma,z) (\alpha z^{\alpha-1} \ast e _k( z^{\alpha})')^{-\ast} \ast (f\mid_{\mathbb E_{\bf i}}^{\beta})' (z), \quad \forall z\in\mathbb E_{\bf i}.
\end{align*}
Finally, note that  the mapping  $ q\mapsto (\alpha q^{\alpha-1} \ast e _k( q^{\alpha})' ) ^{-\ast} $ for all $q\in \mathbb E$ is an   intrinsic slice regular function 
\end{proof}

\begin{rem}
According to the previous proposition and if the zero set of $f$ is the empty, then 
\begin{align*}  
\frac{{d}^{\sigma,\beta}}{dq_{\alpha, k, {\bf i}}} f(q) = & \chi_1(\sigma, q) f(q) + \chi_0(\sigma,q)  
  (\alpha q^{\alpha-1} \ast e _k( q^{\alpha})') ^{-\ast} \ast (f ^{\beta})' (q), 
\end{align*}  
for all $q\in \mathbb E$. 
\end{rem}
      
\begin{defn}
Let ${\bf i}\in \mathbb S^2$ then by $\mathfrak D^{\sigma,\beta}_{\alpha, k,{\bf i}}$ we denote the function set formed by   
$f\in \mathcal {SR}(\mathbb E)$ such that the zero set of $f\mid_{\mathbb E_{\bf i}}$ is the empty set and 
\begin{align*}  
\int_{\mathbb E_{\bf i}} \|\frac{{d}^{\sigma,\beta}}{dq_{\alpha, k, {\bf i}}} f \mid_{ \mathbb E_{\bf i}}(z)\|^2 d_z\mu_{\bf i} < \infty. 
\end{align*}

The function set $\mathfrak D^{\sigma,\beta}_{\alpha, k,{\bf i}}$ is  called ${\bf i}-\beta$ generalized fractal-fractional Dirichlet  type set of slice regular functions induced by $\frac{ {d}^{\sigma,\beta}}{ d q _{\alpha, k, {\bf i}}}$ and associated to $\mathbb E$. 
\end{defn}
\begin{rem}
There are several cases to consider on $\mathfrak D^{\sigma,\beta}_{\alpha, k,{\bf i}}$, however, we will restrict our attention to the case $\beta=1$ because $\dfrac{{d}^{\sigma,1}}{dq_{\alpha, k}}$ is a quaterinonic left-linear operator, $\mathfrak D^{\sigma, 1}_{\alpha, k,{\bf i}}$ is a quaternionic right-linear space and the condition on the zero sets of these function are not necessary.
So from now on, we shall write $\dfrac{ {d}^{\sigma}}{dq_{\alpha, k, {\bf i}}} = \dfrac{ {d}^{\sigma, 1}}{dq_{\alpha, k, {\bf i}}}$   and the function set $\mathfrak D^{\sigma }_{\alpha, k,{\bf i}}= \mathfrak  D^{\sigma,1}_{\alpha, k,{\bf i}}$ becomes at quaternionic right-module, which will be called ${\bf i}$-generalized  fractal-fractional Dirichlet type-space of slice regular functions associated to $\mathbb E$. 
\end{rem} 

\begin{defn} 
The quaternionic right-module $\mathfrak  D^{\sigma }_{\alpha, k,{\bf i}}$ is equipped with 
\begin{align*}  
\langle f,g \rangle_{\mathfrak D^{\sigma }_{\alpha, k,{\bf i}}}:= &  \alpha \overline{f( \frac{1}{2})} g(\frac{1}{2}) + \int_{\mathbb E_{\bf i}\setminus [0,1)}  \overline{ \left( \frac{ {d}^{\sigma}}{ d q _{\alpha, k, {\bf i}}} f\mid_{\mathbb E_{\bf i}}(z) \right)}   
\left(\frac{ {d}^{\sigma}}{dz_{\alpha, k, {\bf i}}} \mid_{\mathbb E_{\bf i}} (z) \right) d_z\mu,
\end{align*}
and 
\begin{align*} 
\|f\|_{\mathfrak  D^{\sigma }_{\alpha, k, {\bf i} } }^2  := & \alpha \|f(\frac{1}{2})\|^2 + \int_{\mathbb E_{\bf i}}\| \frac{ {d}^{\sigma}}{d z _{\alpha, k, {\bf i}} } f\mid_{\mathbb E_{\bf i}}(z) \|^2 d_z\mu,
\end{align*}
for all $f,g\in \mathfrak D^{\sigma }_{\alpha, k, {\bf i}}$. 
\end{defn}

Similarly to Section \ref{Sec2} assume that $\chi_0(\sigma,q) = \sigma$ and $\chi_1(\sigma,q) = 1-\sigma$ for all $q\in \mathbb E$ with $\sigma\in (0,1)$. Then

\begin{prop}\label{SRpro1} If $f\in \mathcal{SR}(\mathbb E)$ then 
\begin{align*}
\left(\textrm{exp}[\dfrac{\sigma-1}{\sigma}  e_k (q^{\alpha})] \ast f(q) \right)'  
= & \frac{1}{\sigma} \textrm{exp}[\dfrac{\sigma-1}{\sigma} e_k (q^{\alpha})] \ast( e_k (q^{\alpha})')^{-\ast} \ast \frac{ {d}^{\sigma}}{ d q _{\alpha, k, {\bf i}}} f(q), 
\end{align*}
 for all $q\in \mathbb E$.
\end{prop}
\begin{proof} Note that the mapping $q\mapsto \textrm{exp}[\dfrac{\sigma-1}{\sigma} e_k (q^{\alpha})]$ is an intrinsic slice regular function and the proof follows from direct computations.
\end{proof}

\begin{rem}\label{rem3}
Given  $f \in \mathfrak  D^{\sigma }_{\alpha, k,{\bf i}}$ by Lemma \ref{LemS} there exist $f_1,f_2 \in \mathcal {SR}(\mathbb E)$ such that  $f\mid_{\mathbb E_{\bf i}}(z) = f_1(z) + f_2 (z) {\bf j} $, where ${\bf j}\in\mathbb S^2$ and  
\begin{align}\label{DirichletEsp}
\dfrac{ {d}^{\sigma}}{ d q _{\alpha, k, {\bf i}} } f\mid_{\mathbb E_{\bf i}} (z)  = &  \dfrac{ {d}^{\sigma }}{ d z _{\alpha, k} } f_1(z) +  \dfrac{ {d}^{\sigma }}{ d z _{\alpha, k} } f_2(z)  {\bf j}  , \quad \forall z\in \mathbb E_{\bf i} , \nonumber \\
 \|f\|_{\mathfrak  D^{\sigma }_{\alpha, k, {\bf i} } }^2  = & \|f_1\|_{\mathcal D^{\sigma }_{\alpha, k}}^2    +  \|f_2\|_{\mathcal D^{\sigma }_{\alpha, k}}^2, 
\end{align}
i.e.,  $f_1,f_2 \in \mathcal D^{\sigma }_{\alpha, k}$ in the complex plane $\mathbb C({\bf i})$. In addition, using Theorem \ref{TeoR}   we conclude that  
  $f \in \mathfrak  D^{\sigma }_{\alpha, k,{\bf i} } $  iff   $f_1,f_2 \in \mathcal D^{\sigma }_{\alpha, k}$ in the complex plane $\mathbb C({\bf i})$ and 
  $f = P_{ {\bf  i} }(f_1+ f_2{\bf j}) $ on $\mathbb E$.
\end{rem}

\begin{prop} 
The quaternionic right-module $(\mathfrak D^{\sigma}_{\alpha, k, {\bf i}}, \|\cdot \|_{\mathfrak D^{\sigma}_{\alpha, k, { \bf i}}})$ is a normed quaternionic right-module . 
\end{prop}
\begin{proof} 
Let $f\in \mathfrak  D^{\sigma }_{\alpha, k,{\bf i}}$ such that $\|f\|_{\mathfrak  D^{\sigma }_{\alpha, k, {\bf i}}} = 0$. By the previous remark there exist $f_1,f_2 \in \mathcal D^{\sigma }_{\alpha, k}$ in the complex plane $\mathbb C({\bf i})$ such that 
$f = P_{ {\bf  i} }(f_1+ f_2{\bf j}) $ on $\mathbb E$ and 
$$\|f\|_{\mathfrak  D^{\sigma }_{\alpha, k, {\bf i}}}^2 = \|f_1\|_{\mathcal D^{\sigma }_{\alpha, k}}^2 + \|f_2\|_{\mathcal D^{\sigma }_{\alpha, k}}^2.$$
Therefore,  $\|f_1\|_{\mathcal D^{\sigma }_{\alpha, k}} = \|f_2\|_{\mathcal D^{\sigma }_{\alpha, k}}=0$ and from Proposition \ref{pro2} we see that $f_1\equiv f_2\equiv  0$. Then $f\equiv 0$.   

From direct computations we obtain  that  
\begin{align*}
\|f \|_{\mathfrak  D^{\sigma }_{\alpha, k, {\bf i}} }  \geq  & 0 , \\ 
\|  f a  \|_{\mathfrak  D^{\sigma }_{\alpha, k, {\bf i}  } } = &  \|f \|_{\mathfrak  D^{\sigma }_{\alpha, k,{\bf i}} } |a| \\
\|f+g \|_{\mathfrak  D^{\sigma }_{\alpha, k, {\bf i}} }  \leq & \|f \|_{\mathfrak  D^{\sigma }_{\alpha, k, {\bf i} } }   + \|g \|_{\mathfrak  D^{\sigma }_{\alpha, k, {\bf i}} }  ,  
\end{align*}
for all $f,g \in  \mathfrak  D^{\sigma }_{\alpha, k, {\bf i}}$ and $a\in \mathbb H$.
\end{proof}
 
\begin{prop}   
$(\mathfrak  D^{\sigma }_{\alpha, k, {\bf i}} , \|\cdot \|_{\mathfrak D^{\sigma}_{\alpha, k, {\bf i}}})$ is  a  quaternionic Banach right-module.  
\end{prop}

\begin{proof}
If $(f_n)_{\mathbb N}$ is a Cauchy sequence of elements of  $ \mathfrak  D^{\sigma }_{\alpha, k, {\bf i}}$ then from  Remark \ref{rem3} there exists 
  $f_{1,n}, f_{2,n}   \in \mathcal D^{\sigma }_{\alpha, k}$ in the complex plane $\mathbb C({\bf i})$ for all $n\in \mathbb N$ such that 
  $f_n = P_{ {\bf  i}}(f_{1,n}+ f_{2,n}{\bf j}) $ on $\mathbb E$, where ${\bf j} \in \mathbb S^2$ such that ${\bf j}\perp {\bf i}$ and the identity   
\begin{align*} 
 \|f_n - f_m \|_{\mathfrak  D^{\sigma }_{\alpha, k, {\bf i} } }^2  = & \|f_{1,n} - f_{1,m}\|_{\mathcal D^{\sigma }_{\alpha, k}}^2 
    +  \|f_{2,n} - f_{2,m}\|_{\mathcal D^{\sigma }_{\alpha, k}}^2, 
\end{align*}
for all $n,m\in \mathbb N$ shows that  $(f_{1,n})_{n\in \mathbb N}$,  $ (f_{2,n})_{n\in \mathbb N}$   are Cauchy sequences in
 $ \mathfrak  D^{\sigma }_{\alpha, k } $. Using  Proposition 
\ref{pro3}  and Remark \ref{rem3} there exist $ f_ 1  ,  \ f_{2}\in  \mathfrak  D^{\sigma }_{\alpha, k } $ such that    
$(f_n)_{n\in \mathbb N}$ converges   to $f = P_{ {\bf  i} }(f_1+ f_2{\bf j}) \in  \mathfrak  D^{\sigma }_{\alpha, k, {\bf i}}$.
\end{proof}

\begin{rem}
It is clearly that $\langle \cdot ,\cdot  \rangle_{\mathfrak  D^{\sigma }_{\alpha, k, {\bf i}}}$ is a quaternionic  inner product in  $\mathfrak  D^{\sigma }_{\alpha, k, {\bf i}}$. Therefore 
$$\left( \mathfrak  D^{\sigma }_{\alpha, k, {\bf i}}, \langle \cdot ,\cdot \rangle_{\mathfrak D^{\sigma }_{\alpha, k, {\bf i}}}\right)$$ is a quaternionic Hilbert right-module.
\end{rem}

\begin{prop}
Let $f \in \mathcal {SR}(\mathbb B) \cap \mathfrak  D^{\sigma }_{\alpha, k, {\bf i}}$ and assume there exists a sequence of quaternions $(a_n)_{n\geq 0}$ such that  
$$f(q)= \sum_{n=0}^\infty q^n a_n ,\quad \forall q\in \mathbb B,$$
then 
\begin{align*}  
& \|f\|_{\mathfrak  D^{\sigma }_{\alpha, k, {\bf i} } }^2 = \alpha   |f (\frac{1}{2})|^2 + (1-\sigma)^2 \pi \sum_{n=0}^\infty  
\frac{1}{n+1} | a_{ n}|^2   \\
  &  +  \frac{\sigma^2 }{2\alpha^2 }  \sum_{n,m=0}^\infty  (n+1)(m+1)\left(a_{n+1} \bar a_{m+1}                                    
		-  {\bf i }a_{n+1} \bar a_{m+1} {\bf i}\right)  \alpha_{m,n}
    \\ 
&  + \frac{  (1-\sigma) \sigma }{2\alpha} \sum_{n,m=0}^\infty 2 (m+1)  Re \left[  
 \left(a_{m+1} \bar a_n -  {\bf i }a_{m+1} \bar a_n {\bf i} \right) \beta_{n,m} \right],   
\end{align*}
where  
 \begin{align*}
 \alpha_{m,n}=  &    \lim_{
 {   \begin{array}{c} 
  \rho \to 1^-   \\ 
  \theta_1\to -\pi^+\\
  \theta_2 \to  \pi^- \end{array}}}    
\int_{0}^\rho \int_{\theta_1}^{\theta_2}  
\frac{r ^{n+m+3 -2\alpha } e^{i \theta (n-m)}} {| e_{k-1} (r^{\alpha} e^{ i \alpha \theta}) |^2} dr d\theta  \\
\beta_{m,n} = & 
 \lim_{
 {\begin{array}{c} 
  \rho \to 1^-   \\ 
  \theta_1\to -\pi^+ \\
  \theta_2 \to  \pi^- \end{array}}}    
\int_{0}^\rho \int_{\theta_1}^{\theta_2} \frac{ r^{n+m+ 2-\alpha }  e^{i\theta ( m +1 -n - \alpha)}}{e_{k-1}(r^{\alpha} e^{i \theta \alpha})} dr d\theta
\end{align*}
\end{prop}

\begin{proof}
Given ${\bf j}\in \mathbb S^2$ and ${\bf j}\perp {\bf i}$ there exist $a_{k,n}\in \mathbb C({\bf i})$ for all $n\geq 0$ and $k=1,2$ such that $a_n = a_{1,n}+ a_{2,n}{\bf j}$. Then $f_k(z) = \sum_{n=0}^{\infty} z^n a_{k,n}$ for all $z\in \mathbb E_{\bf i}$ and $k=1,2$ such that $f\mid_{\mathbb E_{\bf i}} = f_1 + f_2{\bf j}$ on $\mathbb E_{\bf i}$. 

In addition, $f_{1}, f_{2} \in \mathcal D^{\sigma }_{\alpha, k}$  and
\begin{align*}
 \|f\|_{\mathfrak  D^{\sigma }_{\alpha, k, {\bf i} } }^2  = & \|f_1\|_{\mathcal D^{\sigma }_{\alpha, k}}^2    +  \|f_2\|_{\mathcal D^{\sigma }_{\alpha, k}}^2 . 
 \end{align*}   
To finish the proof we use Proposition \ref{pro6} and the identities:
  $$  {\bf i }a_n {\bf i}= - a_{1,n}+ a_{2,n} {\bf j} ,  \quad
 \dfrac{1}{2}\left(  a_n -  {\bf i }a_n {\bf i}\right)  =   a_{1,n}    , \quad  \dfrac{1}{2}\left(  a_n +   {\bf i }a_n {\bf i}\right)  =   a_{2,n} \quad \forall  n\geq  0,$$
to obtain that 
 \begin{align*}
 & a_{1, n } \bar a_{1, m } + a_{2, n } \bar a_{2, m } \\
 = & 
 \dfrac{1}{2}\left(  a_n -  {\bf i }a_n {\bf i}\right)  \dfrac{1}{2}\left( \bar a_m -  {\bf i } \bar  a_m {\bf i}\right)  + 
 \dfrac{1}{2}\left(  a_n +  {\bf i }a_n {\bf i}\right)  \dfrac{1}{2}\left( \bar a_m + {\bf i } \bar a_m {\bf i}\right)  \\
 = &  
 \dfrac{1}{4}\left[a_n \bar a_m -  {\bf i }a_n {\bf i} \bar a_m - a_n {\bf i } \bar a_m {\bf i} - {\bf i }a_n \bar a_m {\bf i}    + 
    a_n \bar a_m +  {\bf i }a_n {\bf i} \bar a_m  +  a_n  {\bf i } \bar a_m {\bf i}  - {\bf i }a_n \bar a_m {\bf i}   \right]  \\
    = &  \dfrac{1}{2}\left(a_n \bar a_m  -  {\bf i }a_n \bar a_m {\bf i}     \right)  , \quad \forall \  n,m\geq 0.
\end{align*} 
\end{proof}

\begin{cor}
Let $f \in \mathcal {SR}(\mathbb B)$ and a sequence of quaternions $(a_n)_{n\geq 0}$ such that  
 $$f(q)= \sum_{n=0}^\infty q^n a_n ,\quad \forall q\in \mathbb B.$$
Then for any $f\in  \mathfrak  D^{\sigma }_{\alpha, 1, {\bf i}}$ we have  
 \begin{align*}  & \|  f \|^2_{\mathfrak  D^{\sigma }_{\alpha, 1, {\bf i}}} =  \alpha |f(\frac{1}{2})|^2 + 
 (1-\sigma)^2 \pi \sum_{n=0}^\infty  \frac{1}{n+1} | a_{n}|^2     
 + \frac{\sigma^2 }{\alpha^2} \sum_{n=0}^\infty  \frac{(n+1)^2}{2n +4-2\alpha} |a_{ n+1}| ^2  \\
 & + \frac{(1-\sigma) \sigma }{ \alpha} \sum_{n,m=0}^\infty 
 \frac{(m+1) \Im  \left[ ( e^{i2\pi (m-n+1 -\alpha)}-1 ) \left(a_{m+1} \bar a_n - {\bf i }a_{m+1} \bar a_n {\bf i} \right) \right]}{ (n+m+3-\alpha) (m-n+1-\alpha )}.    
\end{align*}
On the other hand, if $f\in  \mathfrak  D^{\sigma }_{\alpha, \infty, {\bf i}}$.  Then  
\begin{align*}  
& \|  f\|^2_{\mathfrak  D^{\sigma }_{\alpha, \infty, {\bf i}}} = \alpha |f(\frac{1}{2})|^2 + (1-\sigma)^2 \pi \sum_{n=0}^\infty  \frac{1}{n+1} | a_n|^2  \\
  &  +  \frac{\sigma^2}{2 \alpha^2} \sum_{n,m=0}^\infty (n+1)(m+1) \left(a_{n+1} \bar a_{m+1} - {\bf i }a_{n+1} \bar a_{m+1} {\bf i}     \right) \alpha_{n,m} \\ 
&  +
 \frac{  (1-\sigma) \sigma }{\alpha}   \sum_{n,m=0}^\infty   (m+1) \Re  \left[ \left(    a_{m+1} \bar a_n -  {\bf i }a_{m+1} \bar a_n {\bf i} \right) \beta_{n,m} \right],    
\end{align*}
where 
\begin{align*}
\alpha_{n,m}= & \int_{0}^1 \int_{-\pi}^{  \pi}  
 \frac{r^{n+m+3 -2\alpha } e^{i \theta (n-m)}} {\textrm{exp} ( 2r^{\alpha} \cos  (\alpha \theta))} dr d\theta \\
\beta_{n,m} =  &  \int_{0}^1 \int_{-\pi}^{\pi} \frac{ r^{n+m+ 2-\alpha }  e^{ i\theta ( m +1 -n - \alpha)} }{\textrm{exp} (r^{\alpha} e^{ i \theta \alpha})} dr d\theta
\end{align*} 
\end{cor}
\begin{proof}
The first fact follows from the identity $\|f\|^2_{\mathfrak  D^{\sigma }_{\alpha, 1, {\bf i}}} = \|f_1\|^2_{\mathcal D^{\sigma }_{\alpha, 1}} + \|f_2\|^2_{\mathcal D^{\sigma }_{\alpha, 1}}$, Corollary \ref{corK1Infty} and similar computations to the previous theorem such as 
\begin{align*}  
& \|  f \|^2_{\mathfrak  D^{\sigma }_{\alpha, 1, {\bf i}}}  =  \alpha   |f(\frac{1}{2})|^2 +  (1-\sigma)^2 \pi \sum_{n=0}^\infty  \frac{1}{n+1} | a_{n}|^2 +  \frac{\sigma^2 }{\alpha^2 }  \sum_{n=0}^\infty  \frac{(n+1)^2}{ 2n +4-2\alpha} |a_{ n+1}| ^2  \\
&  + \frac{  (1-\sigma) \sigma }{\alpha}   \sum_{n,m=0}^\infty \frac{ 2 (m+1)  \Im  \left[ (e^{i2\pi (m-n+1 -\alpha)} -1) \left( a_{1, m+1} \bar a_{1,n}  +  a_{2, m+1} \bar a_{2,n} \right) \right]}{(n+m+3-\alpha)(m-n+1-\alpha)}.    
\end{align*} 
Beside, the second fact is a consequence of  
$\|  f\|^2_{\mathfrak  D^{\sigma }_{\alpha, \infty, {\bf i}} }  = \|  f_1\|^2_{\mathcal D^{\sigma }_{\alpha, \infty }}  + \|  f_2\|^2_{\mathcal D^{\sigma }_{\alpha, \infty }}  $,  Corollary \ref{corK1Infty} and similar computations. 
\end{proof}

\begin{prop} 
Let $f  \in \mathcal {SR}(\mathbb B) \cap \mathfrak  D^{\sigma }_{\alpha, k, {\bf i}}$ and   ${\bf j} \in \mathbb S^2$. Then 
\begin{align*}   
\|f\|_{\mathfrak  D^{\sigma }_{\alpha, k, {\bf j} } }^2  \leq 8 \|f\|_{\mathfrak  D^{\sigma }_{\alpha, k, {\bf i} } }^2.  
\end{align*}
\end{prop}

\begin{proof}
Representation Formula Theorem allows us to see that 
 \begin{align*}  \frac{ {d}^{\sigma }}{ d q _{\alpha, k, {\bf j} } } f\mid_{\mathbb E_{\bf j}} (w)    =& \frac{1}{2}(1-{\bf j} {\bf i})  \frac{ {d}^{\sigma }}{ d q _{\alpha, k, {\bf i} } }f \mid_{\mathbb E_{\bf i}} (z) 
 +  \frac{1}{2}(1+{\bf j} {\bf i})  \frac{ {d}^{\sigma }}{ d q _{\alpha, k, {\bf i} } } f\mid_{\mathbb E_{\bf i}} (\bar z),
 \end{align*}
 for all $w= x+y{\bf j} \in \mathbb B_{\bf j}$, where $z= x+y{\bf i} \in \mathbb B_{\bf i} $.  Therefore, 
 \begin{align*}  \| \frac{ {d}^{\sigma }}{ d q _{\alpha, k, {\bf j} }  } f 
\mid_{\mathbb E_{\bf j}} (w)  \|^2   \leq &\left(  \|  \frac{ {d}^{\sigma }}{ d q _{\alpha, k, {\bf i} } } f
\mid_{\mathbb E_{\bf i}} (z)\| +  
  \| \frac{ {d}^{\sigma }}{ d q _{\alpha, k, {\bf i} } } f 
\mid_{\mathbb E_{\bf i}}(\bar z)\|  \right)^2 \\
 \leq &4 \left(  \|  \frac{ {d}^{\sigma }}{ d q _{\alpha, k, {\bf i} } } f
\mid_{\mathbb E_{\bf i}} (z)\|^2 +  
  \| \frac{ {d}^{\sigma }}{ d q _{\alpha, k, {\bf i} }} f 
\mid_{\mathbb E_{\bf i}} (\bar z)\|^2  \right), \\
   \int_{\mathbb E_{\bf j}}  \| \frac{ {d}^{\sigma }}{ d q _{\alpha, k, {\bf j} } } f 
\mid_{\mathbb E_{\bf j}} (w)  \|^2  d_w\mu   \leq   &4 \left( \int_{\mathbb E_{\bf i}} 
     \|  \frac{ {d}^{\sigma }}{ d q _{\alpha, k, {\bf i} } } 
\mid_{\mathbb E_{\bf i}}f (z)\|^2 d_z\mu  +  
 \int_{\mathbb E_{\bf i}}  \| \frac{ {d}^{\sigma }}{ d q _{\alpha, k, {\bf i}}} f
\mid_{\mathbb E_{\bf i}} (\bar z)\|^2 d_{\bar z}\mu  \right) \\
  \leq   &8   \int_{\mathbb E_{\bf i}}   \|  \frac{ {d}^{\sigma }}{ d q _{\alpha, k, {\bf i}}} 
\mid_{\mathbb E_{\bf i}}f (z)\|^2 d_z\mu .
\end{align*} 
\end{proof}

\begin{rem} The previous proposition shows that $ \mathcal {SR}(\mathbb B) \cap \mathfrak  D^{\sigma }_{\alpha, k, {\bf i}}$ and 
$\mathcal {SR}(\mathbb B) \cap \mathfrak  D^{\sigma}_{\alpha, k, {\bf j}}$ are the same function sets for any ${\bf i }, {\bf j} \in \mathbb S^2$ but  these are differents as quaternionic  Hilbert right-modules  which  is an usual phenomenon in the theory of slice regular functions
\end{rem}

\begin{cor} Reproducing property.  Given $\sigma\in (0,1)$. Suppose $f\in \mathcal{SR}(\mathbb B)$ such that $\dfrac{ {d}^{\sigma}}{ d q _{\alpha, k,{\bf i}} } f $ belongs to the ${\bf i}-$slice regular Bergman space associated to $\mathbb B_{\bf i}$ for some ${\bf i}\in \mathbb S^2$. Then
\begin{align*}
& f (q)   
 = -  \frac{\sigma }{1-\sigma }  \alpha  q^{ \alpha-1}  \ast (e_{k-1} ( q^{\alpha})    )^{-\ast} \ast   f   '(q)  +  \frac{ 1 }{1-\sigma  }  \int_{\mathbb B_{\bf i}} \mathcal B_{\bf i }(q,\zeta)  \frac{ {d}^{\sigma}}{ d q _{\alpha, k, {\bf i }} } f\mid _{\mathbb B_{\bf i }}(\zeta) d_{\zeta}\mu_{\bf i} ,\\ 
 &  f(q) =  \textrm{exp}[\dfrac{\sigma-1}{\sigma}  ( \  e_k ((\frac{1}{2})^{\alpha}) - e_k (q^{\alpha}) \  ) ] f(\frac{1}{2}) +  \int_{\mathbb B_{\bf i }} \mathcal K_{\frac{1}{2}}(q,\zeta)   \frac{ {d}^{\sigma}}{ d q_{\alpha, k, {\bf i }} } f\mid _{\mathbb B_{\bf i }}(\zeta) d_{\zeta}\mu_{\bf i }, 
\end{align*}
where
$$ \mathcal K_{\frac{1}{2}}(q,\zeta) =  \textrm{exp}[\dfrac{1-\sigma}{\sigma}  e_k (q^{\alpha}) ]  \ast \int_{\gamma_{\frac{1}{2}, q}}  \frac{1}{\sigma}  \textrm{exp}[\dfrac{\sigma-1}{\sigma}  e_k (w^{\alpha}) ]\left(\frac{d}{dw}    e_k (w^{\alpha}) \right)    \mathcal  B_{{\bf i}}(w,\zeta) d_w\mu_{\bf i},$$ 
Here, $\mathcal B_{\bf i }(\cdot, \cdot)$ is the slice regular Bergman kernel associated to $\mathbb B_{\bf i}$ and $\gamma_{\frac{1}{2}, q}$is a rectifiable path contained in $\mathbb B_{I_q} \setminus (-1,0]$ from $\frac{1}{2}$ to $q$.   
\end{cor}

\begin{proof} Use the  reproducing property of the slice regular Bergman kernel,  see \cite{CGSBERgman, CGLSS},   and  Proposition \ref{SRpro1}.
\end{proof}
 
\begin{rem} Similarly to the complex case commented in Remark \ref{RemDirchlet-Bergman} choosing $\sigma=\alpha^2$ then the family of slice regular function spaces $(\mathfrak D^{\alpha^2,1}_{\alpha, k, {\bf i }})_{\alpha \in (0,1]}$ it is possible consider $\alpha =0$.  In addition, for $k=1$ we have that the well-known ${\bf i}$-slice regular function spaces of Dirichlet and Bergman are given $\mathfrak D^{0,1}_{0, 1, {\bf i}} $ and $\mathfrak  D^{1,1}_{1, 1, {\bf i }} $, respectively, i.e.,  these  slice regular function spaces are elements of the family $(\mathfrak  D^{\alpha^2,1}_{\alpha, 1, {\bf i}})_{\alpha \in [0,1]}$.
\end{rem}

\begin{rem}
There are many questions that can be resolved in future papers on the  quaternionic Hilbert right-module $\mathcal {SR}(\mathbb B) \cap \mathfrak  D^{\sigma }_{\alpha, k, {\bf i}}$ that arise from the generalization of the slice regular function  space  of Dirichlet and Bergman and the properties of these two  quaternionic right-modules.
\end{rem}

\section*{Declarations}
\subsection*{Funding} Research partially supported by Instituto Polit\'ecnico Nacional (grant numbers SIP20241638, SIP20241237) and CONAHCYT.
\subsection*{Competing Interests} The authors declare that they have no competing interests regarding the publication of this paper.
\subsection*{Author contributions} All authors contributed equally to the study, read and approved the final version of the submitted manuscript.
\subsection*{Availability of data and material} Not applicable
\subsection*{Code availability} Not applicable
\subsection*{ORCID}
\noindent
Jos\'e Oscar Gonz\'alez-Cervantes: https://orcid.org/0000-0003-4835-5436\\
Carlos Alejandro Moreno-Mu\~noz: https://orcid.org/0009-0001-4040-0290\\
Juan Bory-Reyes: https://orcid.org/0000-0002-7004-1794

\end{document}